\documentclass[a4paper, 11pt]{article}
\usepackage[mathscr]{eucal}
\usepackage{amssymb}
\usepackage{latexsym}
\usepackage{amsthm}
\usepackage{amsmath}
\usepackage[dvips]{graphicx}
\usepackage{psfrag}
\usepackage{a4wide}

\newcommand{\R}{\mathbb{R}}

\newcommand{\A}{\boldsymbol{A}}
\newcommand{\B}{\boldsymbol{B}}
\newcommand{\E}{\boldsymbol{E}}
\newcommand{\J}{\boldsymbol{J}}
\newcommand{\I}{\boldsymbol{I}}
\newcommand{\D}{\boldsymbol{D}}
\newcommand{\F}{\boldsymbol{F}}
\newcommand{\M}{\boldsymbol{M}}

\newcommand{\Sm}{\boldsymbol{S}}
\newcommand{\X}{\boldsymbol{X}}
\newcommand{\x}{\boldsymbol{x}}
\newcommand{\y}{\boldsymbol{y}}

\newcommand{\jv}{\boldsymbol{j}}
\newcommand{\vv}{\boldsymbol{v}}
\newcommand{\uv}{\boldsymbol{u}}

\newcommand{\el}{\boldsymbol{\ell}}

\newcommand{\la}{\langle}
\newcommand{\ra}{\rangle}

\newtheorem{theorem}{Theorem}[section]

\newtheorem{lemma}[theorem]{Lemma}

\theoremstyle{definition}

\newtheorem{remark}[theorem]{Remark}

\makeatletter
\@addtoreset{equation}{section}

\makeatother

\title{Pseudo-Euclidean representations of switching classes of Johnson and Hamming graphs with minimal dimension}

\author{
Hiroshi Nozaki\thanks{Department of Mathematics Education, 
	Aichi University of Education, 
	1 Hirosawa, Igaya-cho, 
	Kariya, Aichi 448-8542, 
	Japan. {\tt hnozaki@auecc.aichi-edu.ac.jp}}
\and 
Masashi Shinohara\thanks{Faculty of Education, 
	Shiga University,
	2-5-1 Hiratsu, Otsu, Shiga 520-0862, 
	Japan.  
	{\tt shino@edu.shiga-u.ac.jp}}
\and
Sho Suda\thanks{Department of Mathematics, National Defense Academy of Japan, Yokosuka, Kanagawa 239-8686, Japan.
{\tt ssuda@nda.ac.jp}}
}

\begin{document}
\maketitle

\renewcommand{\thefootnote}{\fnsymbol{footnote}}

\begin{abstract}
This paper considers minimum-dimensional representations of graphs in pseudo-Euclidean spaces, where adjacency and non-adjacency relations are reflected in fixed scalar square values. 
   A representation of a simple graph $(V,E)$ is a mapping $\varphi$ from the vertices to the pseudo-Euclidean space $\mathbb{R}^{p,q}$ such that
 \[
 ||\varphi(u)-\varphi(v)||=\begin{cases}
 a \text{ if $(u,v) \in E$}, \\
 b \text{ if $(u,v) \not\in E$ and $u\ne v$},\\
 0 \text{ if $u=v$}\\
 \end{cases}
 \]
for some $a,b \in \mathbb{R}$, where \[||\boldsymbol{x}||=\langle \langle \boldsymbol{x},\boldsymbol{x} \rangle \rangle=
\sum_{i=1}^p x_i^2- \sum_{j=1}^q x_{p+j}^2
\] is the scalar square of $\boldsymbol{x}$ in $\mathbb{R}^{p,q}$.   
 For a finite set $X$ in $\mathbb{R}^{p,q}$, 
 we define 
 \[
 A(X)=\{||\boldsymbol{x}- \boldsymbol{y} || \colon\, \boldsymbol{x},\boldsymbol{y} \in X, \boldsymbol{x} \ne \boldsymbol{y}\}. 
 \]
 A finite set $X$ in $\mathbb{R}^{p,q}$ is called an $s$-indefinite-distance set
 if $|A(X)|=s$ holds. 
 An $s$-indefinite-distance set in $\mathbb{R}^{p,0}=\mathbb{R}^p$ is called an $s$-distance set.  
    Graphs obtained from Seidel switching of a Johnson graph sometimes admit Euclidean or pseudo-Euclidean representations in low dimensions relative to the number of vertices. 
    For instance, Lison\v{e}k~(1997) obtained a largest 2-distance set in $\mathbb{R}^8$ and largest spherical 2-indefinite-distance sets in $\mathbb{R}^{p,1}$ for each $p\geq 10$ from the switching classes of Johnson graphs. In the present paper, we consider graphs in the switching classes of Johnson and Hamming graphs of diameter 2 and classify those that admit representations in $\mathbb{R}^{p,q}$ with the smallest possible dimensionality $p+q$ among all graphs in the same class. 
    The method not only recovers known results, such as the largest 2-(indefinite)-distance sets constructed by Lison\v{e}k, but also provides a unified framework for determining the minimum dimension of representations for entire switching classes of strongly regular graphs.
   
\end{abstract}

\bigskip


\section{Introduction}
Representing graphs in geometric spaces is a classical theme in combinatorics and geometry. In particular, representations in Euclidean or pseudo-Euclidean spaces, where scalar square values encode adjacency relations, have been studied in connection with distance sets and extremal configurations \cite{I11,M19,N19,NS16,NS18,R10}. This paper focuses on such representations for graphs arising from Seidel switching of Johnson and Hamming graphs, and aims to determine the minimum dimensions required for such embeddings.

 A {\it representation} of a simple graph $(V,E)$ is a mapping $\varphi$ from the vertices to the pseudo-Euclidean space $\R^{p,q}$ such that
 \[
 ||\varphi(u)-\varphi(v)||=\begin{cases}
 a \text{ if $(u,v) \in E$}, \\
 b \text{ if $(u,v) \not\in E$ and $u\ne v$},\\
 0 \text{ if $u=v$}\\
 \end{cases}
 \]
 for some $a,b \in \mathbb{R}$. 
 Here $||\boldsymbol{x}||=\langle \langle \boldsymbol{x},\boldsymbol{x} \rangle \rangle$ and the bilinear form is defined as  
\[
\langle \langle \x, \y \rangle \rangle=x_1 y_1 +\cdots +x_p y_p-x_{p+1}y_{p+1}-\cdots -x_{p+q}y_{p+q}
\] 
for $\x=(x_1,\ldots,x_{p+q})^\top, \y=(y_1,\ldots,y_{p+q})^\top \in \mathbb{R}^{p+q}$. 
The function $||\boldsymbol{x} - \boldsymbol{y}||$ dose not define a distance unless $\mathbb{R}^{p,q}$ is Euclidean, that is, $\mathbb{R}^{d,0}$.  
In general, the values $a$ and $b$ may take non-positive values.  
When either $a = 0$ or $b = 0$, the mapping may not be injective.
 Our primary concern lies in determining the minimal dimensionality $p+q$ required for a representation with $a \ne b$. 
 This is closely related to the problem of identifying the largest $2$-indefinite-distance sets. 
 
 A subset $X$ of $\mathbb{R}^{p,q}$ ({\it resp}.\ $\mathbb{R}^d$) is called an {\it $s$-indefinite-distance set} ({\it resp}.\ {\it $s$-distance set}) if the size of 
 \[
A(X)=\{|| \x-\y||  \colon\, \x, \y  \in X, \x \ne \y \}
\]
is equal to $s$. 
An $s$-indefinite-distance set is said to be \emph{spherical} if it is on the sphere $\{\x \in \R^{p,q} \mid \langle \langle \x, \x \rangle \rangle =1\}$. 
When $0 \not \in A(X)$, there exists a natural upper bound $|X| \leq \binom{p+q+s}{s}$ for an $s$-indefinite-distance set $X\subset \mathbb{R}^{p,q}$ \cite{BBS83, BT, NSSpre}. 
One of the major problems concerning $s$-indefinite-distance sets is to determine the maximum possible size of such a set for given $s$ and $(p,q)$. 
For a spherical $s$-indefinite-distance set in $\mathbb{R}^{p,q}$, there also exists an upper bound $|X|\leq \binom{p+q+s-1}{s}$ for $p=1$ or $q=1$ \cite{BT}, and 
$|X| \leq \binom{p+q+s-1}{s}+\binom{p+q+s-2}{s-1}$ otherwise \cite{NSSpre}.  

The known (spherical) $2$-indefinite-distance sets that attain the current upper bounds arise as representations of graphs in the switching classes of Johnson graphs \cite{L97}.  
These sets have been obtained only in $\mathbb{R}^8$, and in the spherical case within pseudo-Euclidean spaces $\mathbb{R}^{p,1}$ for $p \geq 10$.

Seidel switching is an operation performed on a simple graph $(V,E)$, where the adjacency relation of each edge between $U$ and $V \setminus U$ is reversed for a given subset $U \subset V$.  
The switching class of a graph is the set of all graphs obtained by performing Seidel switching with respect to any subset $U \subset V$.
A natural question is whether there exist other graphs in the switching class of a Johnson graph such that these graphs have representations as $2$-indefinite-distance sets attaining the bound $|X| \leq \binom{d+2}{2}$. 

Lison\v{e}k \cite{L97} constructs the largest $2$-indefinite-distance sets
from Johnson graphs by assigning suitable values to the distances
in a graph obtained through specific switching operations,
and then computing the required dimension.
However, this approach relies on particular choices of the subset $U$
and the distance assignments,
and does not explain how such choices should be made in general
to obtain low-dimensional representations.
The present paper clarifies the principles behind selecting candidates for $U$
and the corresponding distances
that lead to low-dimensional representations
for strongly regular graphs.

    The {\it minimum dimensionality} of the switching class of a simple graph is defined as the smallest possible value of $p+q$ such that a graph in the switching class admits a representation in $\mathbb{R}^{p,q}$. 
In this paper, we present a general method for determining the minimum dimensionality $p+q$ in representations of all graphs in the switching class of a strongly regular graph $G$.  
Using this method, we determine the minimum dimensionalities of the switching classes of Johnson and Hamming graphs; see Tables~\ref{tab:John} and~\ref{tab:Ham}.

This paper is organized as follows. In Section \ref{sec:2}, we review the basic terminology and fundamental results required for our study. 
In Section~\ref{sec:3}, for a regular graph $G$, we investigate conditions under which the dimensions of representations of graphs in the switching class of $G$ can be determined solely from the spectral data of $G$. 
If a subset $U\subset V$ creates an equitable partition $\{U,V\setminus U\}$, we can explicitly calculate the dimension of the resulting graph from switching with respect to $U$. 
Section \ref{sec:4} addresses determining the minimum dimensionality $d=p+q$ of the switching class of Johnson and Hamming graphs. 
While our method does not yield any new examples of $2$-indefinite-distance sets attaining the upper bound, it successfully reproduces the known largest sets constructed by Lisoněk in $\mathbb{R}^d$ for $d=4,5,8$, and in the spherical case $\mathbb{R}^{p,1}$ for $p \geq 10$.

\begin{table}[h]
    \centering
 {\small    
\begin{tabular}{c|cccccccc}
   $m$  &   4&5&6&7&8&9&10 &$m\geq 11$  \\ \hline
   $|J(m,2)|$ &6&10&15&21&28&36&45&$m(m-1)/2$ \\
    $d=p+q$ & 2&4&5&6&7&8&8&$m-1$\\
    $(p,q)$ & $(2,0)$, $(1,1)$ &$(4,0)$&$(5,0)$&$(6,0)$&$(7,0)$&$(8,0)$&$(8,0)$&$(m-1,0)$, $(m-2,1)$
\end{tabular}
}
    \caption{Minimum dimensionality of the switching class of Johnson graph $J(m,2)$}
    \label{tab:John}
\end{table}

\begin{table}[h]
    \centering
 {\small    
\begin{tabular}{c|cccc}
   $m$  &   2&3&4& $m\geq 5$  \\ \hline
   $|H(2,m)|$ &4&9&16&$m^2$ \\
    $d=p+q$ & 1&4&5&$2m-2$\\
    $(p,q)$ & $(1,0)$ &$(4,0)$&$(5,0)$&$(2m-2,0)$, $(2m-3,1)$
\end{tabular}
}
    \caption{Minimum dimensionality of the switching class of Hamming graph $H(2,m)$}
    \label{tab:Ham}
\end{table}

\section{Preliminaries} \label{sec:2}
Let $G=(V,E)$ be a simple graph of order $n$.  
The {\it adjacency matrix} $\A$ of $G$ is the symmetric $(0,1)$-matrix indexed by $V$ with $(u,v)$-entries
\[
\A(u,v)=\begin{cases}
1 \text{ if $(u,v) \in E$},\\
0 \text{ otherwise}.
\end{cases}
\]
Let $\overline{\A}$ be the adjacency matrix of the complement graph $\overline{G}$. 
 A matrix $\D (a,b)=a \A+b \overline{\A}$ is called a {\it dissimilarity matrix} on $G$ for $a, b \in \R$.  
 A dissimilarity matrix $\D(a,b)$ on $G$ is said to be {\it representable in $\mathbb{R}^{p,q}$} if there exists a subset $X \subset \mathbb{R}^{p,q}$ such that $\D(X) = \D(a,b)$.
Here, $\D(X)$ denotes the matrix $(||\boldsymbol{x} - \boldsymbol{y}||)_{\boldsymbol{x}, \boldsymbol{y} \in X}$. 
 This set $X$ is called a {\it representation of $\D(a,b)$}. 
 
Let $\X$ be an $n\times (p+q)$ matrix whose rows
 are the coordinates of the $n$ points of $X\subset \R^{p,q}$. 
The {\it dimensionality} of $\D(a,b)$ is the least ${\rm rank}(\X)$ among all representations $X$ of $\D(a,b)$. 
Let $\I$ denote the identity matrix and $\jv$ the all-ones vector. 
For a symmetric matrix $\M$ and a vector $\el$ with $\el^\top \jv=1$, let $\F_{\M}(\el)$ denote the symmetric matrix defined by
\[
\F_{\M}(\el)=- (\I-\jv \el^{\top})\M (\I-\el \jv^{\top}). 
\]
Note that the matrix $\I-\el \jv^{\top}$ is a projection matrix onto 
$\jv^\perp$. 
By direct calculation, we can show the equality 
\[
\F_{\D(X)}(\el)=
(-2\la\la \x-\uv,\y-\uv\ra \ra )_{\x,\y \in X}, \]
 where $\uv=
\sum_{\x \in X} \el_{\x}\x$ and $\el_{\x}$ is the $\x$-th entry of $\el$. 

The pair $(r,s)$ is called the {\it signature}  of symmetric matrix $\M$, where $r$ ({\it resp}.\ $s$) is the number of positive ({\it resp}.\ negative) eigenvalues of $\M$ counting multiplicities, 
and write ${\rm sign}(\M)=(r,s)$. 
Gower~\cite{G85} showed that the signature of $\F_{\M}(\el)$ remains constant for all vectors $\el$ satisfying $\el^\top \jv = 1$, and derived the following theorem. 
\begin{theorem}[Theorems 8 and 9 in \cite{G85}]
Let $(p,q)$ be the signature of $\F_{\D (a,b)}(\el)$.  Then, the dimensionality of $\D(a,b)$ is $p+q$.  Moreover, 
the dimension $(s,t)$ of any representation of $\D (a,b)$  satisfies $s \geq p$ and $t \geq q$.  
\end{theorem}  

The signature of $\F_{\D (a,b)}(\el)$ can be calculated
by the eigenvalues and their main angles of  $\D (a,b)$ as Theorem~\ref{thm:dim_PMP} below. 
Let $\M$ be a real symmetric matrix of order $n$. Let $\tau_1,\ldots, \tau_r$ be the distinct eigenvalues of $\M$. 
Let $E_i$ be the eigenspace of $\tau_i$, and let $\E_i$ be the orthogonal projection matrix onto $E_i$, for $i = 1, \ldots, r$.
The {\it main angle} of $\tau_i$ (or $E_i$) is defined to be the value
\[
\beta_i=\frac{1}{\sqrt{n}}\sqrt{ (\E_i \jv)^\top (\E_i \jv)}. 
\]
It is noteworthy that $0\leq \beta_i \leq 1$ and $\sum_{i=1}^r \beta_i^2=1$ hold, and $\beta_i=0$ if and only if $E_i \subset \jv^\perp$. 
An eigenvalue $\tau_i$ is {\it main} if $\beta_i \ne 0$. 
For main eigenvalue $\tau_i$,  it follows that
\[
n\beta_i^2=\max_{\vv \in E_i: \vv^\top \vv=1} (\vv^\top \jv)^2=
\max_{\vv \in E_i}\frac{(\vv^\top \jv)^2}{\vv^\top \vv}
\]
and 
\begin{equation} \label{eq:mainangle}
\frac{1}{n\beta_i^2}= \min_{\vv \in E_i} \frac{\vv^\top \vv}{(\vv^\top \jv)^2}=
\min_{\vv \in E_i:\, \vv^\top \jv=1} \vv^\top \vv.  
\end{equation}

\begin{theorem}[{\cite{NSSpre}}] \label{thm:dim_PMP}
Let $\M$ be a real symmetric matrix with distinct main eigenvalues
$\lambda_1, \ldots,\lambda_r$.
Let $\beta_i$ be the main angle of $\lambda_i$ for $i \in \{1,\ldots, r\}$. 
Let $(p,q)$ be the signature of $\M$.
Let $E_0$ be the eigenspace associated with the eigenvalue $0$ of $\M$, and set $E_0 = \emptyset$ if $0$ is not an eigenvalue of $\M$.   Then, the signature of $\F_{\M}(\el)$ is determined as follows. 
\begin{enumerate}
\item If $E_0 \subset \jv^\perp$ and $\sum_{i=1}^r \beta_i^2/\lambda_i=0$, 
then the signature of $\F_{\M}(\el)$ is $(q-1,p-1)$. 
\item If $E_0 \subset \jv^\perp$ and $\sum_{i=1}^r \beta_i^2/\lambda_i>0$, 
then the signature of $\F_{\M}(\el)$ is $(q,p-1)$. 
\item If $E_0 \subset \jv^\perp$ and $\sum_{i=1}^r \beta_i^2/\lambda_i<0$, 
then the signature of $\F_{\M}(\el)$ is $(q-1,p)$. 
\item If $E_0 \not\subset \jv^\perp$, then the signature of $\F_{\M}(\el)$ is $(q,p)$.
\end{enumerate}

\end{theorem} 

Let $(V,E)$ be a connected $k$-regular graph of order $n$. 
A graph $(V,E)$ is a {\it strongly regular graph}
if there exist $\lambda, \mu \in \mathbb{Z}$ 
such that 
the number of common neighbours of any pair $u, v \in V$ with $(u,v) \in E$ is $\lambda$, 
and that of any pair $u, v \in V$ with $(u,v) \not \in E$ is $\mu$. 
Let $\A$ be the adjacency matrix of a strongly regular graph $(V,E)$. 
The algebra generated by $\A$ over $\mathbb{C}$ is called a {\it Bose--Mesner algebra} $\mathfrak{A}$. The algebra $\mathfrak{A}$ is spanned by $\{\I, \A, \overline{\A}\}$, where $\overline{\A}$ is the adjacency matrix of the complement graph. 
There are only 3 distinct eigenvalues of $\A$, namely $\lambda_0=k$, $\lambda_1$, and $\lambda_2$. 
Let $E_i$ be the eigenspace of $\lambda_i$, and $\E_i$ the orthogonal projection matrix onto $E_i$. 
The algebra $\mathfrak{A}$ is also spanned by $\{\E_0,\E_1,\E_2\}$, which is the primitive idempotents. 
A dissimilarity matrix $\D$ belongs to $\mathfrak{A}$ and can be expressed by $\{\E_0,\E_1,\E_2\}$. 
 Writing $\D=a_0\E_0+a_1\E_1+a_2\E_2$, we obtain
\begin{equation} \label{eq:srg_rep}
    \F_{\D} ((1/n) \jv)= a_1 \E_1 +a_2 \E_2
\end{equation}
since $\I-(1/n)\jv \jv^\top=\I-\E_0=\E_1+\E_2$. 

\begin{theorem}[{\cite[Theorem 3.3 and Remark 3.4]{NSSpre}}] \label{thm:dimrep}
Let $(V,E)$ be a strongly regular graph of order $n$, with eigenspaces $E_0, E_1, E_2$. 
Let $\E_i$ be the orthogonal projection matrix onto $E_i$, whose rank is $m_i$. Let $\D$ be a dissimilarity matrix on $(V,E)$.  
Then, the matrix $\F_{\D}((1/n)\jv)$ forms $a_1 \E_1+a_2 \E_2$ for some $a_1,a_2 \in \mathbb{R}$. In particular, the dimensionality of $\D$ is $\epsilon(a_1) m_1 +\epsilon (a_2) m_2$, where $\epsilon(a)=1$ if $a\ne 0$ and $\epsilon(a)=0$ if $a= 0$. 
\end{theorem}

\section{Dimensions of representations of graphs in switching classes} \label{sec:3}
This section considers the dimensionalities of graph representations in a regular graph's switching class.   
For a simple graph $(V, E)$, {\it Seidel switching} (or simply {\it switching}) with respect to a subset $U \subset V$ is the operation that produces a new graph $(V, E')$, where $\{u, v\} \in E'$ if and only if one of the following holds:
\begin{align*}
\{u, v\} \in E & \qquad \text{for } u, v \in U, \\
\{u, v\} \in E & \qquad \text{for } u, v \in V \setminus U, \\
\{u, v\} \notin E & \qquad \text{for } u \in U,\, v \in V \setminus U.
\end{align*}
The {\it switching class} of  $(V,E)$ is the set of all graphs resulting from the switching operation of $(V,E)$ with respect to any subset $U\subset V$. Let $\A$ be the adjacency matrix of $(V,E)$, and $\overline{\A}$ that of its complement graph. 
The {\it Seidel matrix} of $(V,E)$ is $\A-\overline{\A}$. 
For $U\subset V$, let $\I_U$ denote the diagonal matrix indexed by $V$ whose $(u,u)$-entry is $-1$ if $u \in U$ and $1$ otherwise. 
The Seidel matrix of the graph $(V, E')$ obtained by switching with respect to $U$ is given by $\A' - \overline{\A'} = \I_U (\A - \overline{\A}) \I_U$,  
where $\A'$ is the adjacency matrix of $(V, E')$.  
Throughout this paper, we use the symbol $'$ (apostrophe) to denote the corresponding object associated with the graph obtained by switching.

We investigate the relationship between the dimensionalities of
$\D=\D(a,b)=a\A+b\overline{\A}$ and
$\D'=\D'(a,b)=a\A'+b\overline{\A'}$.
We begin by summarizing the main results obtained in this section.

For a regular graph $(V,E)$, determining the dimensionality of $\D'$,
which is equivalent to determining the signature of $\F_{\D'}(\ell)$
via Theorem~\ref{thm:dim_PMP},
requires computing the main eigenvalues of $\D'$ and the corresponding
main angles. 
However, deriving these quantities directly from the spectral information
of $\A$ alone is difficult.
Instead, we consider the matrix $2\D'-(a+b)\J$, noting that
\[
\operatorname{sign}(\F_{2\D'-(a+b)\J}(\ell))
=
\operatorname{sign}(\F_{\D'}(\ell)).
\]

The matrix $2\D-(a+b)\J$ can be written as
\begin{equation}\label{eq:D_S}
2\D-(a+b)\J=\Sm_{\D}-(a+b)\I,
\end{equation}
where
\[
\Sm_{\D}=\D-\overline{\D}, 
\qquad  
\overline{\D}=\overline{\D}(a,b)=a\overline{\A}+b\A.
\]
Since $\A$ and $\Sm_{\D}$ share the same eigenspaces
(Lemma~\ref{lem:eigen_SD}),
the eigenvalues of $2\D-(a+b)\J$ can be determined from
\eqref{eq:D_S}.
Moreover, because $\Sm_{\D}$ and $\Sm_{\D'}$ have the same eigenvalues
(Lemma~\ref{lem:eigen_SD'}),
the eigenvalues of
$2\D'-(a+b)\J=\Sm_{\D'}-(a+b)\I$
can also be determined
(Lemma~\ref{lem:eigen_2D'}).

To determine
$\operatorname{sign}(\F_{\D'}(\ell))
=
\operatorname{sign}(\F_{2\D'-(a+b)\J}(\ell))$
via Theorem~\ref{thm:dim_PMP},
it remains to identify which eigenvalues of
$2\D'-(a+b)\J$ are main and to compute the corresponding main angles.
Equivalent conditions for an eigenvalue of
$2\D'-(a+b)\J$ to be main are given in
Lemmas~\ref{lem:switch_main} and \ref{lem:mu_0_main},
and the associated main angles can be computed using
Lemma~\ref{lem:cal_angle}.

When the number of main eigenvalues of $2\D'-(a+b)\J$ is at most two,
the main angles can be written explicitly by
Lemma~\ref{lem:main_angle2}.
This situation occurs, in particular, for strongly regular graphs,
whose detailed treatment is deferred to the next section.

On the other hand, for a regular graph $(V,E)$ admitting an equitable
partition $\{U,V\setminus U\}$,
the distinct main eigenvalues of $2\D'-(a+b)\J$
can be expressed in terms of the eigenvalues of $\A$
(Lemma~\ref{lem:eigen_2D'} and Theorem~\ref{thm:regular_angle}),
and the number of main eigenvalues is again at most two.

First, the following lemma is on the spectral information of $\Sm_{\D}$. 
\begin{lemma} \label{lem:eigen_SD}
Let $(V,E)$ be a $k$-regular graph of order $n$ with adjacency matrix $\A$. 
 Let $\lambda_0=k,\lambda_1,\ldots,\lambda_r$ be the distinct eigenvalues of $\A$, and  $E_0,E_1,\ldots,E_r$ the corresponding eigenspaces.  
Then, the following hold. 
\begin{enumerate}
\item The eigenspaces of $\A$ 
and $\Sm_{\D}$ are the same. 
\item The distinct eigenvalues of $\Sm_{\D}$ for $E_i$ are 
$\tau_0=(a-b)(2k-n+1)$ and $\tau_i=(a-b)(2\lambda_i+1)$ for $i\in\{1,\ldots,r\}$. 
\end{enumerate}
\end{lemma}
\begin{proof}
The assertion is clear from the equation
$
\Sm_{\D}= (a-b)(2 \A-\J+\I)$. 
\end{proof}
The matrix $\Sm_{\D}$ can be expressed by the forms 
\begin{align*}
\Sm_{\D}&=(a-b)(\A-\overline{\A})\\ 
&=2 \D-(a+b) \J+ (a+b)\I.
\end{align*}
After switching with respect to $U\subset V$, the matrix $\Sm_{\D}$ is changed as    
$\I_U \Sm_{\D} \I_U=(a-b)(\A'-\overline{\A'})=2 \D'-(a+b) \J+ (a+b)\I=\Sm_{\D'}$, where $\D'=a\A'+b\overline{\A'}$ is a dissimilarity matrix on $(V,E')$. 
The following is the relationship of the spectral information of $\Sm_{\D}$ and $\Sm_{\D'}$. 
\begin{lemma}\label{lem:eigen_SD'}
Let $(V,E)$ be a $k$-regular graph. The notation is defined as above.  
Then, the following hold. 
\begin{enumerate}
\item The eigenvalues of  $\Sm_{\D'}$ are the same as the eigenvalues $\tau_i$ of $\Sm_{\D}$, where $\tau_i$ is given in  Lemma~\ref{lem:eigen_SD}.  
\item The eigenspace of $\Sm_{\D'}$ associated with $\tau_i$ is $\I_U E_i=\{\I_U \vv \colon\, \vv \in E_i \}$ for $i\in \{0,1,\ldots, r\}$. 
\end{enumerate}
\end{lemma}
\begin{proof}
The assertions are clear since the matrix $\I_U$ is an orthogonal matrix and $\I_U \Sm_{\D} \I_U=\Sm_{\D'}$.
\end{proof}
The eigenvalues of $\Sm_{\D}$ and $\Sm_{\D'}$ are the same, but its main angles may be changed. 

Note that the signature of $\F_{\D'}(\el)$ is the same as that of $\F_{2\D'-(a+b)\J}(\el)$.
If we can determine the eigenvalues and the main angles of $\Sm_{\D'}-(a+b)\I=2\D'-(a+b)\J$, 
then the signature of $\F_{\D'}(\el)$ is obtained by Theorem~\ref{thm:dim_PMP}. 
The eigenvalues of $2\D'-(a+b)\J$ are easily obtained from Lemma~\ref{lem:eigen_SD'} as follows. 
\begin{lemma} \label{lem:eigen_2D'}
Let $(V,E)$ be a $k$-regular graph. The notation is defined as above. 
Let $\lambda_0=k,\lambda_i$ be the eigenvalues of $\A$, 
and $\tau_i$ those of $\Sm_{\D}$, which are given in  Lemma~\ref{lem:eigen_SD}. 
Then the eigenvalues of $2\D'-(a+b)\J$ are  
\begin{align*}
\mu_0&=\tau_0-(a+b)=(a-b)(2k-n+1)-(a+b),\\ 
\mu_i&=\tau_i-(a+b)=(a-b)(2\lambda_i+1)-(a+b) \text{ for $i\ne 0$,}
\end{align*}
 and the eigenspaces associated with $\mu_i$ are $\I_U E_i$. 
\end{lemma}
\begin{remark} \label{rem:two_D}
    The two matrices $2\D - (a+b)\J$ and $2\D' - (a+b)\J$ have the same eigenvalues by equality \eqref{eq:D_S}, since $\Sm_{\D}$ and $\Sm_{\D'}$ have the same eigenvalues. 
\end{remark}

To determine the main angles of $2\D'-(a+b)\J$, we provide a necessary and sufficient condition 
for the eigenvalues to be main. 
\begin{lemma} \label{lem:switch_main}
Let $(V,E)$ be a $k$-regular graph. The notation is defined as above. 
Then the following are equivalent for $i= 1,\ldots,r$.  
\begin{enumerate}
\item $\mu_i$ is not a main eigenvalue of $2\D'-(a+b)\J$. 
\item $\vv^\top(\I_U\jv) = 0$ for each $\vv \in E_i$.
\item The characteristic vector of $U$ is in 
$E_i^\perp =\bigoplus_{ j \ne i} E_j$. 
\end{enumerate}
\end{lemma}
\begin{proof}
$(1) \Leftrightarrow (2)$: An eigenvalue $\mu_i$ is not a main eigenvalue if and only if  
 $\boldsymbol{w}^\top \jv =0$ for each $\boldsymbol{w} \in \I_U E_i$, which is equivalent to $\vv^\top(\I_U\jv) = 0$ for each $\vv \in E_i$.  

$(2) \Leftrightarrow (3)$: The condition $\vv^\top(\I_U\jv) = 0$ for each $\vv \in E_i$ is equivalent to $\I_U\jv \in E_i^\perp =\bigoplus_{ j \ne i} E_j$. 
If $i \ne 0$ holds, then $\jv \in E_0 \subset \bigoplus_{ j \ne i} E_j$. 
The characteristic vector $\uv$ of $U$ can be expressed by a linear combination of  $\I_U\jv$ and $\jv$, namely $\uv=(1/2)(\jv-\I_U\jv)$.
Thus, $\I_U\jv \in \bigoplus_{ j \ne i} E_j$ and $i \ne 0$ if and only if $\uv \in \bigoplus_{ j \ne i} E_j$. 
\end{proof}
\begin{lemma} \label{lem:mu_0_main}
Let $(V,E)$ be a $k$-regular graph. 
Let $V_1,\ldots, V_t$ be the connected components of $(V,E)$.  
The notation is defined as above. 
Let $U_i=U \cap V_i $. 
Then, the following are equivalent.
\begin{enumerate}
    \item $\mu_0$ is not a main eigenvalue of $2\D'-(a+b)\J$. 
    \item $|U_i|=|V_i|/2$ for each $i \in \{1,\ldots, t\}$. 
\end{enumerate} 
\end{lemma}

\begin{proof}
From Lemma~\ref{lem:eigen_2D'}, $\mu_0$ is not a main eigenvalue 
if and only if $\I_U E_0 \subset \jv^\perp $. 
Let $\jv_{i}$ be the characteristic vector of $V_i$. 
From $E_0={\rm Span}_{\mathbb{R}}\{ \jv_{1},\ldots, \jv_{t} \}$, it follows that  $\I_U E_0 \subset \jv^\perp $ if and only if 
$|U_i|=|V_i|/2$ for each $i \in \{1,\ldots, t\}$. 
\end{proof}

\begin{lemma} \label{lem:cal_angle}
Let $(V,E)$ be a $k$-regular graph. The notation is defined as above. 
If $\mu_i$ is a main eigenvalue, then the main angle $\beta_i$ satisfies
\begin{equation}\label{eq:angle_D'}
\frac{1}{n\beta_i^2}=\min_{\vv \in E_i:\, \vv^\top (\I_U \jv)=1} \vv^\top \vv. 
\end{equation}
\end{lemma}
\begin{proof} 
If $\mu_i$ is a main eigenvalue, then the main angle $\beta_i$ of  $\I_U E_i$ satisfies
\begin{equation*}
\frac{1}{n\beta_i^2}=
\min_{\boldsymbol{w} \in \I_U E_i:\, \boldsymbol{w}^\top \jv=1} \boldsymbol{w}^\top \boldsymbol{w}
=\min_{\vv \in E_i:\, \vv^\top (\I_U \jv)=1} \vv^\top \vv
\end{equation*}
from \eqref{eq:mainangle}.
\end{proof}
If the main eigenvalues of $2\D'-(a+b)\J$ are only $\mu_0,\mu_1$, then the main angles depend only on $|U|$ as follows. 
\begin{lemma}\label{lem:main_angle2}
Let $(V,E)$ be a $k$-regular graph of order $n$.  The notation is defined as above. 
If the main eigenvalues of $2\D'-(a+b)\J$ are only $\mu_0$ and $\mu_1$, then the main angles $\beta_i$ of 
$\mu_i$ satisfy \[
n \beta_0^2=\frac{(n-2|U|)^2}{n},\qquad 
n\beta_1^2=n-\frac{(n-2|U|)^2}{n}.
\]
\end{lemma}
\begin{proof}
First note that $|U|\ne n/2$, otherwise 
 $\jv^\top (\I_U \jv) = 0$, and $\mu_0$ is not a main eigenvalue. 
It follows from \eqref{eq:angle_D'} that 
\begin{align*}
\frac{1}{n \beta_0^2}=  \hat{\jv}^\top \hat{\jv}
=\frac{n}{(n-2|U|)^2},
\end{align*}
where $\hat{\jv}$ is the element of $E_0$ satisfying $\hat{\jv}^\top (\I_U\jv)=1$, that is, $\hat{\jv}=(1/(n-2|U|))\jv$. 
One has $n\beta_1^2=n(1-\beta_0^2)$, which implies the assertion. 
\end{proof}

We aim to identify conditions under which a regular graph $(V,E)$ and a subset $U \subset V$ exist, such that $2\D'-(a+b)\J$ has only two 
main eigenvalues. We will establish that if $\{U,V \setminus U \}$ forms an equitable partition (see the definition below), then the number of main eigenvalues of $2\D'-(a+b)\J$ is
 at most 2. 
Indeed, if a partition $\{U , V\setminus U\}$ is equitable, then the main angles of $\mu_i$ are explicitly written. 
A partition $\pi=\{C_1,\ldots,C_p \}$ of $V$ is {\it equitable} 
if for each $u \in C_i$, the number $b_{ij}=\{v \in C_j \colon\, \{u,v\} \in E\}$ is constant depending only on $i,j$. 
The matrix $\B=(b_{ij})$ is called a {\it quotient matrix} of $\pi$. 
The {\it characteristic matrix} $\boldsymbol{\pi}$ of $\pi$ is the $n\times p$ matrix whose $i$-th column is the characteristic vector of $C_i$. For the adjacency matrix $\A$ of $(V,E)$, 
one has $\A\boldsymbol{\pi}= \boldsymbol{\pi} \B$, and the eigenvalues of $\B$ are eigenvalues of $\A$. 
For an eigenvector $\vv$ of $\A$ corresponding to $\lambda_i$, if $\vv^\top  \boldsymbol{\pi}$ is not the zero vector, then $\vv^\top  \boldsymbol{\pi}$ is a left eigenvector of $\B$ corresponding to $\lambda_i$. 
If $(V,E)$ is a $k$-regular graph, then $k$ is an eigenvalue of $\B$. 
The following theorem shows the main eigenvalues of $2\D'-(a+b)\J$ when $\{U,V\setminus U\}$ is equitable.

\begin{theorem} \label{thm:regular_angle}
Let $(V,E)$ be a $k$-regular graph of order $n$ with connected components $V=V_1\cup \cdots \cup V_t$. 
Let $\pi=\{U, V\setminus U\}$ be an equitable partition of $(V,E)$ with quotient matrix $\B$, where $U\ne \emptyset,V$. Let $U_i=U \cap V_i$. 
Let $\lambda_i$ be an eigenvalue of $(V,E)$, and $\mu_i$ that of  $2\D'-(a+b)\J$ given in Lemma~\ref{lem:eigen_2D'}. 
If $\lambda_0=k$, $\lambda_1$ are the eigenvalues of $\B$,   
then the following hold.
 \begin{enumerate}
\item If $|U_i|\ne |V_i|/2$ for some $i\in \{1,\ldots, t\}$, then the main eigenvalues of $2\D'-(a+b)\J$ are $\mu_0$, $\mu_1$. 
\item If $|U_i|= |V_i|/2$ for each $i \in \{1, \ldots, t\}$, then the main eigenvalue of $2 \D'-(a+b)\J$ is $\mu_1$. 
\end{enumerate}
\end{theorem}
\begin{proof}
Let $\boldsymbol{\pi}$ be the characteristic matrix of $\pi$.
For $i\ne 0$, if $\mu_i$ is a main eigenvalue of $2\D'-(a+b)\J$, then $\vv^\top (\I_U \jv) \ne 0$ for some $\vv \in E_i$ by Lemma \ref{lem:switch_main}. 
If $\vv^\top (\I_U \jv) \ne 0$ for some $\vv=(\vv_x)_{x \in V} \in E_i$, then 
\[
(\vv_x)^\top (\I_U \jv)= -\sum_{x \in U}\vv_x+\sum_{x \in V\setminus U} \vv_x \ne 0,
\]
and $\vv^\top  \boldsymbol{\pi}=(\sum_{x \in U} \vv_x,\sum_{x \in V \setminus U} \vv_x)$ is not the zero vector. 
Then, $\vv^\top  \boldsymbol{\pi}$ is a left eigenvector of $\B$, and hence $\lambda_i$ is an eigenvalue of $\B$. 
Therefore, for $i\ne 0$, the candidate of main eigenvalue $\mu_i$ is only $\mu_1$. 

If $|U_i|\ne |V_i|/2$ for some $i\in \{1,\ldots, t\}$, then $\mu_0$ is a main eigenvalue from Lemma \ref{lem:mu_0_main}. 
If $\mu_0$ is the only main eigenvalue, then $\jv \in \I_U E_0$.  
Since $\jv \not\in \I_U E_0={\rm Span}\{\I_U \jv_1, \ldots, \I_U \jv_t\}$, the eigenvalue $\mu_1$ must be main. 

If $|U_i|= |V_i|/2$ for each $i \in \{1, \ldots, t\}$, then $\mu_0$ is not a main eigenvalue from Lemma \ref{lem:mu_0_main}. The remaining candidate $\mu_1$ must be main. 
\end{proof}
From Lemma~\ref{lem:main_angle2} and Theorem~\ref{thm:regular_angle}, we can determine the signature of $\F_{\D'}(\el)$ (namely the dimensionality of $\D'$) for an equitable partition $\{U,V\setminus U\}$ using main eigenvalues and main angles. 
\section{Minimum dimensionalities for Johnson and Hamming graphs} \label{sec:4}
In this section, 
we determine the minimum dimensionality of the switching class of a given Johnson or Hamming graph. 
Initially, we outline a general method to determine this minimum dimensionality for the switching class of a strongly regular graph. 
Consider a strongly regular graph $(V,E)$ with adjacency matrix $\A$ and eigenspaces $E_0={\rm Span}_{\mathbb{R}}\{ \jv \}$, $E_1$, $E_2$, associated with eigenvalues $\lambda_0=k,\lambda_1,\lambda_2$, respectively.  
Let $\D=\D(a,b)$ be a dissimilarity matrix $\D(a,b)=a \A +b \overline{\A}$ for $a,b \in  \mathbb{R}$.  
For $U\subset V$, the notation $'$ denotes the corresponding object after switching with respect to $U$. 
By Remark~\ref{rem:two_D}, the eigenvalues of $2 \D-(a+b)\J$ and $2 \D'-(a+b)\J$ coincide. 
Consequently, the dimensionality difference of $\D$ and $\D'$ is at most 2 by Theorem~\ref{thm:dim_PMP} (where $\F_{\D}=\F_{2\D-(a+b)\J}$). 
To achieve smaller dimensionalities of $\D'$, we choose distances $a,b \in \mathbb{R}$ such that 
$\D(a,b)$ has a smaller dimensionality. 
Choosing the orthogonal projection matrices $\E_i$ onto $E_i$ as the representations is usually sufficient to determine the minimum dimensionality of the switching class by Theorem~\ref{thm:dimrep}. 
The distances $a$, $b$ of the representation with respect to $E_1$ ({\it resp}. $E_2$) satisfy $b/a=-\lambda_2$ ({\it resp}. $a/b=-\lambda_1-1$) \cite{BB05} for $\lambda_1<\lambda_2$. 
Then, the eigenspace associated with eigenvalue 0 of $2 \D-(a+b)\J$ is $E_2$ or $E_0 \oplus E_2$ from \eqref{eq:srg_rep}.  
From Theorem~\ref{thm:SRG} (1), (i) below, if the characteristic vector $\uv$ of $U$ is not in $E_0\oplus E_1$, 
then the dimensionality of $\D'$ is at least that of $\D$. 
We need to identify all subsets $U$ such that $\uv \in E_0\oplus E_1$, which determines the main eigenvalues of $2 \D'-(a+b)\J$ by Theorem~\ref{thm:SRG} (2), (3), (ii). 
For such $U$, the main angles are calculated using Lemma~\ref{lem:main_angle2}.  
These main eigenvalues and main angles of $2 \D'-(a+b)\J$ determine the dimensionality of $\D'$ by Theorem~\ref{thm:dim_PMP}. 
This method enables the determination of the minimum dimensionality of the switching class of a strongly regular graph. 
Actually, the classification of $U\subset V$ such that $\uv \in E_0\oplus E_1$ is challenging for a general strongly regular graph, but it is possible for Johnson and Hamming graphs. 

The following theorem shows the main eigenvalues of graphs in the switching class of a strongly regular graph. 
\begin{theorem} \label{thm:SRG}
Let $(V,E)$ be a strongly $k$-regular graph. Let $E_0={\rm Span}_{\mathbb{R}}\{ \jv \} $, $E_1$, $E_2$ be the eigenspaces of $(V,E)$.
Let $U$ be a non-empty subset of $V$ with $U\ne V$ and $\uv$ its characteristic vector. 
 Let $\mu_i$ be the eigenvalues of  $2\D'-(a+b)\J$ given in Lemma~\ref{lem:eigen_2D'}, which are the same as those of   $2\D-(a+b)\J$. 
 The notation is the same as that defined at the beginning of this section. 
If the eigenspace associated with eigenvalue 0 of $2\D-(a+b)\J$ is $E_2$, namely $\mu_0,\mu_1 \ne 0, \mu_2=0$, then the following hold. 
\begin{enumerate}
\item If $\uv \not\in E_0 \oplus E_1$ holds,
then the dimensionality of $\D'$ is greater than that of $\D$. 
\item If $\uv \in E_0 \oplus E_1$ and $|U|\ne |V|/2$, then the main eigenvalues of $2\D'-(a+b)\J$
 are $\mu_0$ and $\mu_1$. 
\item If $\uv \in E_0 \oplus E_1$ and $|U|= |V|/2$, then the main eigenvalue of $2 \D'-(a+b)\J$ is $\mu_1$. 
\end{enumerate}
If the eigenspace associated with eigenvalue 0 of $2\D-(a+b)\J$ is $E_0 \oplus E_2$, namely $\mu_1 \ne 0, \mu_0=\mu_2=0$, then the following hold. 
\begin{enumerate}
\item[{\rm (i)}] If $\uv \not\in E_0 \oplus E_1$ or $|U|\ne |V|/2$ holds,
then ${\rm sign}(\F_{\D'}(\el))={\rm sign}(\F_{\D}(\el))$, and the dimensionality of $\D'$ is the same as that of $\D$. 
\item[{\rm (ii)}] If $\uv \in E_0 \oplus E_1$ and $|U|= |V|/2$ holds, then the main eigenvalue of $2 \D'-(a+b)\J$ is $\mu_1$. 
\end{enumerate} 
\end{theorem}

\begin{proof}
First note that by Remark~\ref{rem:two_D},  the 
eigenvalues of $2\D-(a+b)\J$ and $2\D'-(a+b)\J$ are the same; in particular, the signatures are the same. Let $\mathcal{E}_0$ be the eigenspace of $2\D-(a+b) \J$ associated with the eigenvalue 0. 
 
(1): From $\mathcal{E}_0=E_2\subset \jv^{\perp}$ 
and Theorem~\ref{thm:dim_PMP}, 
the dimensionality $p_1+q_1$ of $\D$ is less than $p+q$, where $(p_1,q_1)={\rm sign}(\F_{\D}(\el))={\rm sign}(\F_{2\D-(a+b)\J}(\el))$ and $(p,q)={\rm sign}(2\D-(a+b)\J)$. 
  If $\uv \not\in E_0 \oplus E_1$ holds, then $\mu_2=0$ is a main eigenvalue of $2\D'-(a+b)\J$ by Lemma~\ref{lem:switch_main}. 
 Since the eigenspace $\I_U E_2$ of $2\D'-(a+b)\J$ associated with $\mu_2=0$ is not contained in $\jv^\perp$, we have 
\begin{equation*}
    {\rm sign}(\F_{\D'}(\el))={\rm sign}(\F_{2\D'-(a+b)\J}(\el))={\rm sign}(2\D'-(a+b)\J)={\rm sign}(2\D-(a+b)\J)=(p,q)
\end{equation*}    
by Theorem~\ref{thm:dim_PMP} (4). Therefore, the dimensionality of $\D'$ is $p+q$, greater than that of $\D$. 

(2): If $\uv \in E_0 \oplus E_1$ holds, then $\mu_2=0$ is not a main eigenvalue of $2 \D'-(a+b)\J$ by Lemma~\ref{lem:switch_main}. 
Since a strongly regular graph is connected, if $|U|\ne |V|/2$ holds, 
then $\mu_0$ is a main eigenvalue. 
From $\jv \not\in \I_U \E_0$, the remaining eigenvalue $\mu_1$ is also main. 

(3): If $\uv \in E_0 \oplus E_1$ holds, then $\mu_2=0$ is not a main eigenvalue of $2 \D'-(a+b)\J$ by Lemma~\ref{lem:switch_main}. 
Since a strongly regular graph is connected, if $|U|= |V|/2$ holds, 
then $\mu_0$ is not a main eigenvalue. The remaining eigenvalue $\mu_1$ is main. 

(i): 
From $\mathcal{E}_0=E_0 \oplus E_2 \not\subset \jv^\perp$ and Theorem~\ref{thm:dim_PMP} (4), 
it follows that  \[{\rm sign}(\F_{\D}(\el))={\rm sign}(\F_{2\D-(a+b)\J}(\el))={\rm sign}(2\D-(a+b)\J).\]
 By Lemma~\ref{lem:switch_main}, if $\uv \not\in E_0 \oplus E_1$ or $|U|\ne |V|/2$ holds, then $\mu_0=\mu_2=0$ is a main eigenvalue of $2\D'-(a+b)\J$. 
 Since the eigenvalue $0$ of $2\D'-(a+b)\J$ is main, its eigenspace is not contained in $\jv^{\perp}$. 
By Theorem~\ref{thm:dim_PMP} (4), 
\begin{multline*}
    {\rm sign}(\F_{\D'}(\el)) = {\rm sign}(\F_{2\D'-(a+b)\J}(\el))={\rm sign}(2\D'-(a+b)\J)\\={\rm sign}(2\D-(a+b)\J)={\rm sign}(\F_{\D}(\el)).
\end{multline*}  This is the assertion (i). 

(ii): If $\uv \in E_0 \oplus E_1$ and $|U|= |V|/2$ hold, then neither $\mu_0=0$ nor $\mu_2=0$ 
is a main eigenvalue of $2 \D'-(a+b)\J$. Then the remaining eigenvalue $\mu_1$ is main. 
This implies (ii). 
\end{proof}


A clique $C$ of a strongly $k$-regular graph $(V,E)$ is a {\it Delsarte clique} 
if it attains the upper bound $|C|\leq 1-k/\lambda$ \cite{D73}, where $\lambda$ is the smallest eigenvalue of $(V,E)$.  
The Johnson graph $J(m,2)$ is the graph with vertex set $V = \{ S \subset X \colon\, |S| = 2 \}$ and edge set $E = \{ (S,T) \colon\, |S \cap T| = 1 \}$, where $X$ is a finite set of size $m$.
It is easily verified that for $m\geq 5$, the Delsarte cliques of $J(m,2)$ are $\{S \in V \colon\, i \in S\}$ for any $i \in X$, whose sizes are $m-1$. 
The Delsarte cliques of $J(4,2)$ are $\{S \in V \colon\, i \in S\}$ and $\{S \in V \colon\, i \not\in S\}$ for any $i \in X$. 

The following theorem provides the minimum dimensionalities of the switching classes of Johnson graphs.

\begin{theorem} \label{thm:Johnson}
Let $(V,E)$ be a Johnson graph $J(m,2)$ with $m\geq 4$.  
Then, the following show the minimum dimensionality $d$ of the switching class of $(V,E)$ and the switching sets $U$ that give the graphs achieving it, as well as the distances $a,b$ of the representations. 
Here, $|U| \leq |V|/2$, and $a,b$ ($a\ne b$) are up to scaling. 
\begin{enumerate}
\item For $m=4$, one has $d=2$ and $(p,q)=(m-2,0)=(2,0), (m-3,1)=(1,1)$. 
For $(p,q)=(2,0)$,  $U$ is the empty set and $a=1$, $b=0$, the representation is a regular triangle.
For $(p,q)=(1,1)$,  $U$ is the set of two vertices that are not adjacent and $a=1$, $b=0$. 
\item For $m=5$, one has $d=4$ and $(p,q)=(m-1,0)=(4,0)$. 
For the representations, $U$ is the empty set or a Delsarte clique and $a=1$, $b=2$, or
the graph induced by $U$ is a 5-cycle and $a=2$, $b=1$. 
\item For $m=6,7,9$, one has $d=m-1$ and $(p,q)=(m-1,0)$. 
For the representations, $U$ is the empty set or a Delsarte clique and distances $a=1$, $b=2$. 
\item For $m= 8$, one has $d=7$ and $(p,q)=(m-1,0)=(7,0)$. 
For the representations, $U$ is any set and distances $a=1$, $b=2$. 
\item For $m=10$, one has $d=8$ and $(p,q)=(m-2,0)=(8,0)$.   For the representations, $U$ is a Delsarte clique and distances $a=1$, $b=2$. 
\item For $m\geq 11$, one has $d=m-1$, and $(p,q)=(m-2,1),(m-1,0)$. 
For $(p,q)=(m-1,0)$, $U$ is the empty set , and distances $a=1$, $b=2$. 
For $(p,q)=(m-2,1)$, $U$ is a Delsarte clique, and  distances $a=1$, $b=2$.
\end{enumerate} 
\end{theorem}
\begin{proof}
The spectrum of a graph consists of its eigenvalues and their multiplicities. 
For the Johnson graph $J(m,2)$, the spectrum is $\{2(m-2)^{(1)},(m-4)^{(m-1)}, -2^{(m(m-3)/2)}\}$, and the corresponding eigenspaces are denoted by $E_0$, $E_1$, and $E_2$, respectively.  By Theorem~\ref{thm:dimrep}, for $m\geq 5$, the minimum dimensionality of $(V,E)$ is $m-1$, obtained from the representation $\D$ with respect to $E_1$.  
  The signature of the representation $\F_{\D}(\el)$ is $(m-1,0)$, and the distances $a=1$, $b=2$.  
  For the other distances $a,b$ up to scaling, $\F_{\D(a,b)}(\jv)$ can be expressed by $\F_{\D(a,b)}(\jv)=c_1 \E_1+c_2 \E_2$ with non-zero $c_2$.  
  From \eqref{eq:srg_rep}, the eigenspace of $\D(a,b)$ associated with eigenvalue 0 has dimension at most $m= \dim E_0 +\dim E_1$.   
  Then, the dimensionality of $\D'(a,b)$ is at least $m(m-3)/2-2$ by Theorem \ref{thm:dim_PMP}. This is larger than the dimensionality $m-1$ of $\D(1,2)$ for $m\geq 6$. 
For $m=4,5$, we have to consider the representations with respect to $E_2$, which is treated later.  
For the distances $a=1,b=2$, we consider $\D'=\D'(1,2)$ that is obtained by switching with respect to $U\subset V$. 

We determine the family $\mathcal{U}$ of all non-empty subsets $U\subset V$ such that $2|U|\leq |V|$ and its characteristic vector $\uv$ is in $E_0 \oplus E_1$. 
If the dimensionality of $\D'$ is smaller than
that of $\D$, then $U$ should be in $\mathcal{U}$. 
Recall that the vertex set $V$ consists of all $2$-element subsets of $\{1,\ldots, m\}$, and for each $i \in \{1,\ldots, m\}$, the set $U_i = \{v \in V \colon\, i \in v\}$ forms a Delsarte clique. 
Let $\uv_i$ be the characteristic vector of $U_i$. 
It is known that $E_0 \oplus E_1={\rm Span}_{\mathbb{R}}\{\uv_1,\ldots, \uv_m\}$ \cite{M03}.  
If the characteristic vector $\uv$ of some $U\subset V$ is in $E_0 \oplus E_1$, then $\uv$ can be expressed by $\uv=\sum_{i=1}^m a_i \uv_i$. 
By comparing the $\{i,j\}$-entry of $\uv$ with $\{i,j\} \in V$, we obtain  
\begin{equation} \label{eq:a_i}
a_i+a_j=\begin{cases}
1 \text{ if $\{i,j\} \in U$},\\
0 \text{ if $\{i,j\} \not\in U$}.
\end{cases}
\end{equation}
The coefficient matrix of any three equations for $\{i,j\},\{j,k\}, \{k,i\} \in V$ in \eqref{eq:a_i} is non-singular. Therefore, the following hold. 
\begin{enumerate}
\item[(i)] If vertices $\{i,j\},\{j,k\},\{k,i\}$ are in $U$, then $a_i=a_j=a_k=1/2$.  
\item[(ii)] If vertices $\{i,j\},\{j,k\}$ are in $U$ and $\{k,i\}$ is not in $U$, then $a_j=1,a_i=a_k=0$. 
\item[(iii)] If a vertex $\{i,j\}$ is in $U$ and $\{j,k\},\{k,i\}$ are not in $U$, then $a_i=a_j=1/2,a_k=-1/2$.  
\item[(iv)] If vertices $\{i,j\},\{j,k\},\{k,i\}$ are not in $U$, then $a_i=a_j=a_k=0$.  
\end{enumerate}
It is noteworthy that $a_i \in \{0,\pm 1/2,1\}$ for any $i$. Since $U$ is not empty, (i), (ii), or  (iii) occurs for some $i,j,k$. 
If (iii) occurs,  
then there exists $k$ such that $a_k=-1/2$. 
Since the case $a_k=-1/2$ is only in (iii), for such $k$ and any distinct $i,j$ with $i,j\ne k$, we have $\{j,k\}, \{k,i\}, \not\in U$, $\{i,j\} \in U$ and $a_i=a_j=1/2$. This situation implies $U=V \setminus U_k$. 
If (i) occurs and (iii) does not occur, then (ii) does not occur and $U=V$. 
If (ii) occurs, then $a_j=1$ for some $j$, and $a_i=0$ for any $i\ne j$, which implies  $U=U_j$. 
Therefore,  $\mathcal{U}$ is the set of Delsarte cliques $\{U_1,\ldots,U_m\}$.

The eigenvalues of $2 \D'-3\J$ are  
\begin{equation} \label{eq:eigen_2D-3J}
\mu_0=\frac{1}{2}(m-1)(m-8), \qquad  
\mu_1=-2(m-2), \qquad \mu_2=0    
\end{equation}
from Lemma \ref{lem:eigen_2D'}. 
The eigenspace associated with 0 of $\D=\D(1,2)$ is $E_2$ for $m\ne 8$, and $E_0 \oplus E_2$ for $m=8$. 

For $m=8$, if $U$ is in $\mathcal{U}$, then $U$ is a Delsarte clique and $|U|=m-1\ne |V|/2$. Thus, ${\rm sign}(\F_{\D}(\el)) = {\rm sign}(\F_{\D'}(\el))$ for each $U\subset V$ by Theorem~\ref{thm:SRG} (i). 
The signature of $\F_{\D'}(\el)$ is $(m-1,0)$ for any $U$, that proves the statement (4). 

For $m\ne 8$, 
if $U \not\in \mathcal{U}$ holds, then
the dimensionality of $\D'$ is greater than $\D$ by Theorem~\ref{thm:SRG} (1). 
For $m=4$, $|U|=|V|/2$ holds for $U \in \mathcal{U}$, and the main eigenvalue of $2\D'-3 \J$ is only $\mu_1<0$ by Theorem \ref{thm:SRG} (3). 
By Theorem~\ref{thm:dim_PMP}, ${\rm sign} (\F_{\D'}(\el))=(3,0)$.  
For $m\ne 4,8$, the main eigenvalues are $\mu_0$ and $\mu_1$ by Theorem \ref{thm:SRG} (2) and 
the main angles are  
\[
|V|\beta_0^2=\frac{1}{2m}(m-1)(m-4)^2, \qquad 
|V|\beta_1^2=\frac{4}{m}(m-1)(m-2), \qquad
\beta_2^2=0
\] 
by Lemma~\ref{lem:main_angle2}. 
From \eqref{eq:eigen_2D-3J}, the signature of the matrix $2\D'-3 \J$ is 
\[
\begin{cases}
(0,m) \text{ if $5 \leq m \leq 7$},\\
(1,m-1) \text{ if $m\geq 9$}.   
\end{cases}
\]
From 
\[
|V|\sum_{i=0}^1\frac{\beta_i^2}{\mu_i}=
-1+\frac{2}{m-8}\begin{cases}
<0 \text{ if $5 \leq m \leq 7$ or $m\geq 11$},\\
=0 \text{ if $m = 10$},\\
>0 \text{ if $m= 9$}
\end{cases}
\]
and Theorem~\ref{thm:dim_PMP}, 
the signature of $\F_{\D'}(\el)$ is 
\[
\begin{cases}
(m-1,0) \text{ if $5 \leq m \leq 7$ or $m=9$},\\
(m-2,0)\text{ if $m = 10$},\\
(m-2,1) \text{ if $m\geq 11$}.   
\end{cases}
\]
This implies the statements $(3)$, $(5)$, and $(6)$. 

Finally, we treat the representation $\D$ with respect to $E_2$ for $m=4,5$. 
We classify non-empty subsets $U$ of $V$ satisfying $2|U| \leq |V|$ and whose characteristic vectors lie in $E_0 \oplus E_2$. 
Since $E_0 \oplus E_1$ is spanned by the characteristic vectors $\uv_i$ of the Delsarte cliques, it follows that $E_1$ is spanned by $m \uv_i-2\jv$ ($i=1,\ldots,m$). 
We can directly classify such subsets $U$ by determining the characteristic vectors that are orthogonal to $m\uv_i - 2\jv$ for all $i$. 

For $m=4$,  $\D$ has distance $a=1$, $b=0$, which is a triangle in $\mathbb{R}^2$. The eigenspace associated with eigenvalue 0 of $\D$ is $E_1$ from \eqref{eq:srg_rep}. 
A non-empty subset $U$ of $V$ ($2|U|\leq |V|$) whose characteristic vector is  in $E_0 \oplus E_2$ forms $U=\{\{i,j\},\{k,l\}\}$ for distinct $i$, $j$, $k$, $l$. For such $U$, the eigenspace $\I_U E_1$ of $2\D'-\J$ associated with $0$ is contained $\jv^\perp$. 
By Theorem \ref{thm:SRG} (2), $\mu_0=2$ and $\mu_2=-4$ are main. Since $\beta_0^2=1/9$ and $\beta_2^2=8/9$, we have $\beta_0^2/\mu_0+\beta_2^2/\mu_2=-1/6<0$.  
By Theorem \ref{thm:dim_PMP} (3), the signature of $\F_{\D'}(\el)$ is $(m-3,1)=(1,1)$.  
This implies the statement $(1)$. 

For $m=5$, $\D$ has distance $a=2$, $b=1$ and ${\rm sign}(\F_{\D}(\el))$ is $(5,0)$. Here $\mu_0=\mu_1=0$ and ${\rm sign}(2\D- 3 \J )=(0,5)$. 
To apply Theorem~\ref{thm:SRG} (ii), we would like to find $U\subset V$ whose characteristic vector $\uv$ is in $E_0\oplus E_2$ and $|U|=|V|/2=5$. If $U$ does not satisfy the condition, ${\rm sign}(\F_{\D}(\el))={\rm sign}(\F_{\D'}(\el))$. Indeed,  $U
$ satisfies the condition if and only if $U$ is a 5-cycle. For such $U$, the eigenspace $\I_U E_1$ of $\D'$ associated with $0$ is contained $\jv^\perp$, and $\mu_2=-6$ is the main eigenvalue of $2\D'-3 \J$. Since $\beta_2^2/\mu_2=-1/6<0$, one has ${\rm sign}(\F_{\D'}(\el))=(4,0)$ by Theorem \ref{thm:dim_PMP} (3). 
This implies the statement $(2)$. 
\end{proof}

\begin{remark}
Theorem \ref{thm:Johnson} (5) for $(p,q)=(8,0)$ corresponds to a largest $2$-distance set in $\mathbb{R}^8$ \cite{L97}. 
Theorem \ref{thm:Johnson} (6) for $(p,q)=(m-2,1)$ corresponds to 
 largest spherical 2-indefinite-distance sets  in $\mathbb{R}^{m-2,1}$ for $m\geq 12$ \cite{L97}. 
For $m=11$, the representation in $\mathbb{R}^{9,1}$ is not spherical. 
Indeed, the main eigenvalues of $2\D'=(2\D'-3\J)+3\J$ can be calculated by Theorem 2.2 in \cite{NS16}, namely they are the roots of the polynomial 
\[
(\mu_0-x)(\mu_1-x)\left(1+3\left( \frac{|V|\beta_0^2}{\mu_0-x}+\frac{|V|\beta_1^2}{\mu_1-x}\right) \right)=x^2-2(m-2)^2
x -2(m-1)(m-2)(m-11).\] 
Thus, ${\rm sign}(\D')=(2,m-2)$ for $m\geq 12$ and ${\rm sign}(\D')=(1,m-2)$ for $m = 11$. 
From Remark~5.2 and Theorem 5.3 in \cite{NSSpre}, 
the embedding is spherical for $m\geq 12$ and not spherical for $m=11$. 
\end{remark}

\begin{remark}
For Theorem \ref{thm:Johnson} (4), we can take any subset $U$ to obtain the minimum dimensionality. 
Indeed, the representations are subsets of the antipodal 3-distance set $X$ with 56 points in the sphere $S^6$. 
The set $X$ attains the absolute bound $|X| \leq 2 \binom{d+1}{2}$ for antipodal spherical $s$-distance set, and it is the tight spherical 5-design \cite{DGS77}. 
The representation $Y$ is a half set of $X$, that satisfies $X=Y\cup (-Y)$ and $2|Y|=|X|$. 
Switching on a half set of an antipodal set corresponds to transforming an element $x$ of $U$ into $-x$.
\end{remark}

\begin{remark}
For Theorem \ref{thm:Johnson} (2) $m=5$, the complement graph of $J(5,2)$ is the Petersen graph. 
Actually, for a 5-cycle $U$, the graphs $\overline{\A}$ and $\A'$ are isomorphic. 
The representation is the largest 2-distance set in $\mathbb{R}^4$ \cite{L97}. 
\end{remark}


The Hamming graph $H(2,m)$ is the graph with vertex set 
$V = X \times X$ and edge set 
$E = \{ ((x_1, y_1), (x_2, y_2)) \colon\, (x_1 = x_2 \text{ and } y_1 \ne y_2) \text{ or } (x_1 \ne x_2 \text{ and } y_1 = y_2) \}$, 
where $X$ is a finite set of size $m$. 
It is easily verified that the Delsarte cliques of $H(2,m)$, whose sizes are $m$, are 
$\{(x,y) \colon\, y \in X\}$ and $\{(y,x) \colon\, y \in X\}$ for any $x \in X$. 
The following theorem provides the minimum dimensionalities of the switching classes of Hamming graphs.

\begin{theorem} \label{thm:Hamming}
Let $(V,E)$ be a Hamming graph $H(2,m)$ with $m\geq 2$.  
Let 
\[
t_m=\frac{m}{2}\left(1-\sqrt{\frac{m-4}{m-2}}\right)
\]
for $m\geq 5$. 
Then, the following show the minimum dimensionality $d$ of the switching class of $(V,E)$ and the switching sets $U$ that give the graphs achieving it, as well as the distances $a,b$ of the representations. 
Here, $|U| \leq |V|/2$, and $a,b$ ($a\ne b$) are up to scaling. 
\begin{enumerate}
\item For $m=2$, one has $d=2m-3=1$ and $(p,q)=(2m-3,0)$. 
For the representations, $U$ is one vertex, and 
the distances are $a=1$, $b=0$, or $U$ is two adjacent vertices, the distances are $(a,b)=(1,0),(2,1)$. 
\item For $m=3$, one has $d=2m-2=4$ and $(p,q)=(2m-2,0)=(4,0)$. 
For  the representations, $U$ is the empty set or a Delsarte clique with $3$ vertices, and the distances are $a=1$, $b=2$,  
or $U$ is the empty set or a coclique on $3$ vertices, and the distances are $a=2$, $b=1$. 
\item For $m=4$, one has $d=2m-3=5$ and $(p,q)=(2m-3,0)=(5,0)$. 
For the representations, $U$ is a disjoint union of two Delsarte cliques, $|U|=2m$, and the distances are $a=1$, $b=2$. 
\item For $m\geq 5$, one has $d=2m-2$ and $(p,q)=(2m-2,0)$, $(2m-3,1)$.  
For the representations, $U$ is a disjoint union of $s$ Delsarte cliques with $0 \leq s\leq \lfloor m/2 \rfloor$, $|U|=sm$, and the distances are $a=1$, $b=2$. 
If $0\leq s<t_m$, then $(p,q)=(2m-2,0)$.   
If $t_m <s \leq  \lfloor m/2 \rfloor$, then $(p,q)=(2m-3,1)$. Note that $t_m$ is not an integer. 
\end{enumerate} 
\end{theorem}
\begin{proof}
For $m=2$, $H(2,2)$ is the 4-cycle. 
The graphs obtained from switching are the 4-cycle, the claw graph, and three isolated vertices. 
Observing these graphs, we can easily determine the minimum dimensionality of the switching class is 1. 
If $U$ is the set of one vertex, then the graph obtained from the switching is the claw graph, and  
the representation achieving the smallest dimensionality is two points in $\R^1$ with $a=1$, $b=0$ or three points in $\R^1$ with $a=1$, $b=2$. 
If $U$ is the set of two adjacent vertices, then the graph obtained from the switching is the 4-cycle, and the representation achieving the smallest dimensionality is two points in $\R^1$ with $a=1$, $b=0$. This implies the assertion $(1)$. 

We suppose $m> 2$. 
The spectrum of a Hamming graph $(V,E)$ is $\{2(m-1)^{(1)},(m-2)^{(2m-2)}, -2^{((m-1)^2)}\}$, and the corresponding eigenspaces are denoted by $E_0$, $E_1$, and $E_2$, respectively.  By Theorem~\ref{thm:dimrep}, the minimum dimensionality of $(V,E)$ is $2m-2$, which is obtained from the representation with respect to $E_1$ for $m\geq 3$. The signature of the representation $\F_{\D}(\el)$ is $(2m-2,0)$, and the distances $a=1$, $b=2$.  
  For the other distances $a,b$ up to scaling, $\F_{\D(a,b)}(\jv)$ can be expressed by $\F_{\D(a,b)}(\jv)=c_1 \E_1+c_2 \E_2$ with non-zero $c_2$.  
  From \eqref{eq:srg_rep}, the eigenspace of $\D(a,b)$ associated with eigenvalue 0 has dimension at most $m= \dim E_0 +\dim E_1$.   
  Then, the dimensionality of $\D'(a,b)$ is at least $(m-1)^2-2$ by Theorem \ref{thm:dim_PMP}. This is larger than the dimensionality $2m-2$ of $\D(1,2)$ for $m\geq 4$.  We treat the representation with respect to $E_2$ for $m=3$ later. 
For the distances $a=1,b=2$, we consider $\D'=\D'(1,2)$ that is obtained by switching with respect to $U\subset V$. 

We determine the family $\mathcal{U}$ of all non-empty subsets $U\subset V$ such that $2|U| \leq |V|$ and its characteristic vector $\uv$ is in $E_0 \oplus E_1$. 
The vertex set $V$ is $X \times X$, where $X=\{1,\ldots, m\}$, and $U_i=\{(i,v) \in V \colon\,  v \in X\}$ and $V_i=\{(v,i) \in V \colon\,  v \in X\}$ for $i \in\{1,\ldots,m\}$ are the Delsarte cliques. 
Let $\uv_i$ ({\it resp.}\ $\vv_i$) be the characteristic vector of $U_i$ ({\it resp.}\ $V_i$). 
It is known that $E_0 \oplus E_1={\rm Span}_{\mathbb{R}}\{\uv_1,\ldots, \uv_m, \vv_1,\ldots,\vv_m\}$ \cite{M03}.  
If the characteristic vector $\uv$ of some subset $U\subset V$ is in $E_0 \oplus E_1$, then $\uv$ can be expressed by $\uv=\sum_{i=1}^m a_i \uv_i+\sum_{i=1}^m b_i \vv_i$, whose coefficients satisfy 
\begin{equation} \label{eq:a_i2}
a_i+b_j=\begin{cases}
1 \text{ if $(i,j) \in U$},\\
0 \text{ if $(i,j) \not\in U$}.
\end{cases}
\end{equation}
We will show that $\mathcal{U}$ consists of $U_I=\bigcup_{i \in I} U_i$ and $V_I=\bigcup_{i \in I} V_i$ for any index set $I\subset X$ with $|I| \leq \lfloor m/2 \rfloor$. 

Suppose there exist $k,l\in X$ such that $a_1+b_k=1$ ($(1,k) \in U$), $a_1+b_l=0$ ($(1,l) \not\in U$). From these equations, it follows that $b_l+1=b_k$, and hence $a_i+b_k=1$ ($(i,k) \in U$) and $a_i+b_l=0$ ($(i,l) \not\in U$) for any $i \in X$. 
Therefore, $U=V_I$ such that $I=\{k \colon\, a_1+b_k=1\}$.

Suppose $a_1+b_k=1$ ($(1,k) \in U$) for any $k\in X$. 
The choices of $a_i$ is $a_1$ or $a_1-1$. 
If $a_i=a_1$, then $a_i+b_k=1$ ($(i,k) \in U$) for any $k \in X$. 
If $a_i=a_1-1$, then $a_i+b_k=0$ ($(i,k) \not\in U$) for any $k \in X$. 
Therefore, $U=U_I$ such that $I=\{i \colon\, a_i=a_1\}$. 

Suppose $a_1+b_k=0$ ($(1,k) \not\in U$) for any $k\in X$. 
The choices of $a_i$ is $a_1$ or $a_1+1$. 
If $a_i=a_1$, then $a_i+b_k=0$ ($(i,k) \not\in U$) for any $k \in X$. 
If $a_i=a_1+1$, then $a_i+b_k=1$ ($(i,k) \in U$) for any $k \in X$. 
Therefore, $U=U_I$ such that $I=\{i \colon\, a_i=a_1+1\}$. 

From the assumption $2|U| \leq |V|$, we choose $I$ such that $|I| \leq \lfloor m/2 \rfloor$.
Thus, $\mathcal{U}$ consists of $U_I=\bigcup_{i \in I} U_i$ and $V_I=\bigcup_{i \in I} V_i$ for any index set $I\subset X$ with $|I| \leq \lfloor m/2 \rfloor$.

The eigenvalues of $2\D-3\J$ (which are the same as those of  $2\D'-3\J$) are 
\[\mu_0=m(m-4),\qquad \mu_1=-2m, \qquad \mu_2=0\]
by Lemma \ref{lem:eigen_2D'}. 
The eigenspace of $2\D-3\J$ associated with the eigenvalue 0 is 
$E_2$ for $m\geq 5$ and $m=3$, and $E_0 \oplus E_2$ for $m=4$. 
For $m \geq  5$, $m=3$ ({\it resp}. $m=4$), 
if $U \not\in \mathcal{U}$ holds, then
the dimensionality of $\D'$ is greater than ({\it resp}. equal to) that of $\D$ by Theorem~\ref{thm:SRG} (1) ({\it resp}. (i)). 

Consider the cases where the number of main eigenvalues is 1 and 2, respectively. 

(I) If $m$ is even and $|U|=|V|/2$ holds for $U \in \mathcal{U}$, then the main eigenvalue of $2\D'-3 \J$ is only $\mu_1<0$ by Theorem~\ref{thm:SRG} (3), (ii). 
Since the main eigenvalue is only $\mu_1$, we have 
$\jv \in \I_U E_1$, and the eigenspace of $2\D'-3\J$ associated with $0$ is contained in $\jv^\perp$.  
In this case, ${\rm sign} (\F_{\D'}(\el))=(2m-3,0)$ for $m=4$ and ${\rm sign} (\F_{\D'}(\el))=(2m-3,1)$ for $m>4$ even by Theorem~\ref{thm:dim_PMP} (3).   
For $m=4$ and $U \in \mathcal{U}$ with $|U|\ne |V|/2$, we have ${\rm sign} (\F_{\D'}(\el))=(2m-2,0)$ by Theorem~\ref{thm:SRG} (i). 
Therefore the assertion (3) follows. 

(II) For $m \geq 5$, $m=3$ and $U \in \mathcal{U}$ with $|U|\ne |V|/2$,  we calculate the signature of $\F_{\D'}(\el)$ with the main angles of $\mu_0$ and $\mu_1$. 
For $U_I \in \mathcal{U}$ or $V_I \in \mathcal{U}$ with $|I|=s$, the main angles satisfy  
\[
|V|\beta_0^2=(m-2s)^2, \qquad 
|V|\beta_1^2
=4s(m-s), \qquad
\beta_2^2=0
\] 
by Lemma~\ref{lem:main_angle2}. 
The signature of the matrix $2\D'-3 \J$ is 
\[
\begin{cases}
(0,2m-1) \text{ if $ m=3 $},\\
(1,2m-2) \text{ if $m\geq 5$}.   
\end{cases}
\]
It follows that 
\[
|V|\sum_{i=0}^1\frac{\beta_i^2}{\lambda_i}=
\frac{2s(s-m)(m-2)+m^2}{m(m-4)}\begin{cases}
<0 \text{ if $m=3$, or $t_m<s <  m/2 $ and $m\geq 5$},\\
=0 \text{ if $s=t_m$}, \\
>0 \text{ if $0 <s<t_m$ and $m \geq 5$},  
\end{cases}
\]
where 
\[
t_m=\frac{m}{2}\left(1 - \sqrt{\frac{m-4}{m-2}} \right). 
\]

We prove that $\sqrt{(m-4)/(m-2)}$ is irrational, and that the value $t_m$ is not an integer.
Assume, for contradiction, that there exist integers $k, l \in \mathbb{Z}_{>0}$ such that $(m-4)/(m-2) = (k/l)^2$. This assumption is valid since we may take $k, l$ to be positive for $m \geq 5$. 
The greatest common divisor of $m-2$ and $m-4$ is 1 if $m$ is odd, and it is 2 if $m$ is even.  If $m$ is odd, then $k^2=m-4$ and $l^2=m-2$. 
There do not exist positive integers $k,l$ such that $2=l^2-k^2=(l+k)(l-k)$ holds. 
 If $m$ is even, then $k^2=(m-4)/2$ and $l^2=(m-2)/2$. 
There do not exist positive integers $k,l$ such that $1=l^2-k^2=(l+k)(l-k)$ holds. 

By Theorem~\ref{thm:dim_PMP}, 
the signature of $\F_{\D'}(\el)$ is 
\[
\begin{cases}
(2m-2,0) \text{ if $m =3$},\\
(2m-3,1)\text{ if $t_m<s < m/2 $ and $m\geq 5$},\\
(2m-2,0) \text{ if $0< s<t_m$ and $m \geq 5$}.   
\end{cases}
\]

Combining the cases (I) and (II) together with the case $s=0$, we obtain 
\[
\begin{cases}
(2m-3,1)\text{ if $t_m<s \leq \lfloor m/2 \rfloor $ and $m\geq 5$},\\
(2m-2,0) \text{ if $0\leq  s<t_m$ and $m \geq 5$}.   
\end{cases}
\]
This implies the assertion (4). 

For $m=3$, we consider the representation with respect to $E_2$, which has distances $a=2$, $b=1$. 
Actually, $H(2,3)$ is a self-complementary graph.  
The representations with respect to $E_1$ and $E_2$ are isomorphic and the assertion (2) follows. 
\end{proof}

\begin{remark}
For Theorem~\ref{thm:Hamming} (3), the representation is the largest 2-distance set in $\mathbb{R}^5$ \cite{L97}. 
The corresponding graph is known as the complement of the Clebsch graph. 
\end{remark}

\bigskip

\noindent
\textbf{Acknowledgments.} 
H. Nozaki was partially supported by JSPS KAKENHI Grant Numbers JP22K03402 and JP24K06688. 
M. Shinohara was partially supported by JSPS KAKENHI Grant Number JP22K03402. 
S. Suda was partially supported by JSPS KAKENHI Grant Numbers JP22K03402 and JP22K03410.

\quad 

\noindent
\textbf{Declarations}

\quad 

\noindent 
\textbf{Conflict of Interest.}
The authors state
that there is no conflict of interest.

\end{document}